\documentclass[12pt,oneside]{article}
\usepackage{amsmath,amssymb,amsfonts,amsthm}
\usepackage{color}
\usepackage[all]{xy}
\usepackage{graphicx}
\usepackage{color}
\usepackage{makeidx}
\usepackage{multirow}
\usepackage{lscape}
\usepackage{rotating}
\newcommand{\T}{{\cal T}}

\newcommand{\Real}{\mathbb R}

\newcommand{\To}{\longrightarrow}

\newcommand {\cppp}{\mathfrak{X}(TM)}

\def\Section#1{\vspace{30truept}\addtocounter{section}{1}\setcounter{thm}{0}\setcounter{equation}{0}
{\noindent\Large\bf\arabic{section}.~~#1}\par \vspace{12pt}}

\newtheorem{thm}{Theorem}[section]
\newtheorem{cor}[thm]{Corollary}
\newtheorem{lem}[thm]{Lemma}
\newtheorem{prop}[thm]{Proposition}
\newtheorem{defn}[thm]{Definition}

\newtheorem{rem}[thm]{Remark}

\numberwithin{equation}{section}
\newcommand\undersym[2]{\raisebox{-7pt}{\tiny$#2$}{\kern-8pt}\mbox{$#1$}}
\newcommand\undersymm[2]{\raisebox{-8pt}{\tiny$#2$}{\kern-15pt}\mbox{$#1$}}
\newcommand\overast[1]{\raisebox{9pt}{\small$\ast$}{\kern-9pt}\mbox{$#1$}}
\newcommand\overlind[1]{\raisebox{10pt}{\small$\overline{{\hspace{2pt}}\star}$}{\kern-7.5pt}\mbox{$#1$}}
\newcommand\overlinc[1]{\raisebox{10pt}{\tiny$\overline{{\hspace{2pt}}\circ}$}{\kern-7.5pt}\mbox{$#1$}}
\newcommand\overlina[1]{\raisebox{10pt}{\small$\overline{{\hspace{1pt}}\ast}$}{\kern-7.5pt}\mbox{$#1$}}
\newcommand\overcirc[1]{\raisebox{10pt}{\tiny{$\circ$}}{\kern-7.5pt}\mbox{$#1$}}
\newcommand\overdiamond[1]{\raisebox{10pt}{\small$\star$}{\kern-7.5pt}\mbox{$#1$}}

\newcommand\tovercirc[1]{\raisebox{5pt}{\tiny{$\circ$}}{\kern-5.5pt}\mbox{$#1$}}
\newcommand\toverdiamond[1]{\raisebox{5pt}{\tiny$\star$}{\kern-5.5pt}\mbox{$#1$}}
\newcommand\toverast[1]{\raisebox{5pt}{\tiny$\ast$}{\kern-5pt}\mbox{$#1$}}

\makeindex

\begin{document}
\title{\bf{Existence and uniqueness of Chern connection in the Klein-Grifone approach}}%\footnote{ArXiv Number: }}
\author{\bf{ Nabil L. Youssef$^{\,1}$ and S. G.
Elgendi$^{2}$}}
\date{}
%\thanks{\it Department of Mathematics, etc}
%\pagestyle{fancy}
             % End of preamble and beginning of text.
\maketitle                     % Produces the title.
\vspace{-1.16cm}
\begin{center}
{$^{1}$Department of Mathematics, Faculty of Science,\\ Cairo
University, Giza, Egypt}\\
\end{center}

\begin{center}
{$^{2}$Department of Mathematics, Faculty of Science,\\ Benha
University, Benha,
Egypt}
\end{center}

\begin{center}
E-mails: nlyoussef@sci.cu.edu.eg, nlyoussef2003@yahoo.fr\\
{\hspace{1.8cm}}salah.ali@fsci.bu.edu.eg, salahelgendi@yahoo.com
\end{center}

\smallskip
\vspace{1cm} \maketitle
\smallskip
{\vspace{-1.1cm}} \noindent{\bf Abstract.}
The Klein-Grifone approach to global Finsler geometry is adopted. A global existence and uniqueness theorem for Chern connection is formulated and proved. The torsion and curvature tensors of Chern connection are derived. Some properties and the Bianchi identities for this connection are investigated. A concise comparison between Berwald, Cartan and Chern connections is presented.

\vspace{5pt}

\vspace{5pt}
%\newpage
%%%%%%%%%%%%%%%%%%%%%%%%%%%%%%%%%%%%%%%%%%%%%%%%%%%%%%%% %%%%%%%%%%%%%%%%%%%%%%%%%%%%%%%%%%%%%%%%%%%%%%%%%%%%%%%%%%%%%%%%%%%%%%%%%%%%%%%%%%%

%\vspace{30truept}\centerline{\Large\bf{Introduction}}\vspace{12pt}
\Section{Introduction}
The most well known and widely used approaches to global Finsler
geometry are the Klein-Grifone (KG-) approach (\cite{r21},
\cite{r22}, \cite{r27}) and the pullback (PB-) approach \linebreak (\cite{r58}, \cite{r61}, \cite{r74}, \cite{r44}).
Each of the two approaches has its own geometry which differs
significantly from the geometry of the other (in spite of the existence of some links between them).

In the PB-approach, the existence and uniqueness theorems for the four fundamental linear connections (Berwald, Cartan, Chern and Hashiguchi connections) on a Finsler manifold $(M,E)$  have been satisfactorily established (\cite{r92}, \cite{r94}). In the  KG-approach, Grifone \cite{r22} has investigated Cartan and Berwald connections. Szilasi and Vincze \cite{szilasi} have studied Chern
and Hasiguchi connections using the technique of lifting vector fields to the tangent bundle. To the best of our knowledge there is no proof, in the KG-approach, of the existence and uniqueness theorems for Chern and Hashiguchi connections.

In this paper, we formulate and prove an intrinsic existence and uniqueness theorem for Chern connection. We derive the torsion and curvature tensors of this connection. We also study the Bianchi identities and investigate the properties of  the curvature tensors of Chern connection.
Finally, an appendix concerning a global survey of Berwald, Cartan and Chern connections in the KG-approach is presented. It should finally be noted that
the Fr\"{o}licher-Nijenhuis formalism of vector forms \cite{r20} is used extensively in this work.

%%%%%%%%%%%%%%%%%%%%%%%%%%%%%$$$$$$$$$$$$$$$$$$ $$$$$$$$$$$$$$$$$$$$$$$$$$$$$$$$&&&&&&&&&&&&&&&&&&&&&&&&&&&&&&&&&&&&&&&&&&&&&&&&&&

\Section{Notation and Preliminaries}
In this section, we give a brief account of the basic concepts
of the Klein-Grifone approach to global Finsler geometry (\cite{r21},\cite{r22}, \cite{r27}).
The following notations will be used throughout this paper:\\
 $M$: a real differentiable manifold of finite dimension $n$ and
class $C^{\infty}$,\\
  $\mathfrak{F}(M)$: the $\Real$-algebra of differentiable functions
on $M$,\\
 $\mathfrak{X}(M)$: the $\mathfrak{F}(M)$-module of vector fields
on $M$,\\
$\pi_{M}:TM\longrightarrow M$: the tangent bundle of $M$,\\
$\pi: \T M\longrightarrow M$: the subbundle of nonzero vectors
tangent to $M$,\\
$V(TM)$: the vertical subbundle of the bundle $T(TM)$,\\
$\pi^{-1}(TM):= \T M\undersym{\times}{M} TM$: the pullback bundle induced by $\pi$,\\
$ i_{X}$ : the interior derivative with respect to  $X
\in\mathfrak{X}(M)$,\\
$df$ : the exterior derivative  of $f$,\\
$ d_{L}:=[i_{L},d]$, $i_{L}$ being the interior derivative with
respect to a vector form $L$,\\
$\mathcal{L}_X$ : the Lie derivative  with respect to $X\in\mathfrak{{X}}(M)$.

\vspace{6pt}
We have the following short exact sequence of vector bundles
\vspace{-0.1cm}
$$0\longrightarrow
 \pi^{-1}(TM)\stackrel{\gamma}\longrightarrow T(TM)\stackrel{\rho}\longrightarrow
\pi^{-1}(TM)\longrightarrow 0 ,\vspace{-0.1cm}$$
 where the bundle morphisms $\rho$ and $\gamma$ are defined as usual. The vector $1$-form $J$ on
$TM$ defined by $J:=\gamma\circ\rho$ is called the natural almost-tangent structure of $T M$. The vertical vector field ${C}$
on $TM$ defined by ${C}:=\gamma\circ\overline{\eta}$, where $\overline{\eta}$
is the vector field on $\pi^{-1}(TM)$ given by $\overline{\eta}(u)=(u,u)$, is
called the   canonical or the Liouville  vector field.
One can show that the almost-tangent structure $J$ has the following properties:
\begin{equation}\label{J}
J^2=0,\quad [J,J]=0,\quad [C,J]=-J,\quad \text{Im}(J)=\text{Ker}(J)=V(TM).
\end{equation}

A scalar  $p$-form $\omega$ on $TM$ is semi-basic if $ i_{JX}\omega=0,\,\, \forall X\in \cppp$. A vector $\ell$-form $L$ on $TM$ is semi-basic if $ JL=0\,\, \text{and}\,\, i_{JX}L=0, \,\, \forall X\in \cppp$.\\
 A scalar  $p$-form $\omega$ on $TM$ is  homogeneous of degree $r$ if $\mathcal{L}_C\omega=r\omega$. A vector $\ell$-form $L$ on $TM$ is homogeneous of degree $r$ if $[C,L]=(r-1)L$.

 A semi-spray  on $M$ is a vector field $S$ on $TM$, $C^{\infty}$ on $\T M$, $C^{1}$ on $TM$, such that
$JS = C$. A homogeneous semi-spray $S$ of degree $2$
($[{C},S]= S $) is called a spray. If $S$ is a  semi-spray, then
\begin{equation}\label{jx}
J[JX,S]=JX,\,\, \forall X\in \cppp.
\end{equation}

A nonlinear connection on $M$ is a vector $1$-form $\Gamma$ on $TM$,
$C^{\infty}$ on $\T M$, $C^{0}$ on $TM$, such that
$J \Gamma=J, \,\, \Gamma J=-J$.
The vertical and horizontal   projectors $v$\,  and
$h$ associated with $\Gamma$ are defined respectively  by
   $v:=\frac{1}{2}
 (I-\Gamma),\, h:=\frac{1}{2} (I+\Gamma).$
Thus $\Gamma$ gives rise to the direct sum decomposition $T(T M)=
V(TM)\oplus H(TM)$, where $\,\,V(TM):= \text{Im} \, v=\text{Ker} \, h$ is the vertical bundle and   $H(TM):=\text{Im} \, h = \text{Ker}\,
v $ is the horizontal bundle induced by $\Gamma$. An element of $V(TM)$ (resp. $H(TM)$) will be denoted by $vX$  (resp. $hX$).  A nonlinear
connection $\Gamma$ is homogeneous if  $\Gamma$ is homogeneous of degree 1 as a vector form, i.e., $[{C},\Gamma]=0$.
For a homogeneous connection $\Gamma$, we have
 \begin{equation}\label{chx}
 [C,{h}X]={h}[C,X],\,\, \forall X\in \cppp.
 \end{equation}
The torsion $t$ of a nonlinear connection $\Gamma$ is the vector $2$-form  on $TM$ defined by $t:=\frac{1}{2} [J,\Gamma]$.
The curvature $\mathfrak{R}$ of $\Gamma$ is the vector $2$-form on $TM$ defined by $\mathfrak{R}:=-\frac{1}{2}[h,h]$. Given a nonlinear connection $\Gamma$,
an almost-complex structure $F$ $(F^2=-I)$ is  defined by $FJ=h$ and $Fh=-J$. This
$F$ defines an isomorphism of $T_z(TM)$  for all $z\in TM$.
\begin{defn}\label{Finsler}  A Finsler space is a pair $(M,E)$,
where $M$ is a
 differentiable manifold of dimension $n$ and $E$ is a map
 $E: TM \To \Real,$
called energy,  satisfying the axioms{\em:}
 \begin{description}
    \item[\em{\textbf{\em(a)}}] $E(u)>0 $ for all $u\in \T M$ and $E(0)=0$,
    \item[\em{\textbf{\em(b)}}] $E $ is  $C^{\infty}$ on  $\T M$, $C^{1}$ on $TM$,
    \item[\em{\textbf{\em(c)}}]$E$ is homogeneous of degree $2${\em:}
    $\mathcal{L}_{{C}} E=2E$,
    \item[\em{\textbf{\em(d)}}] The exterior  $2$-form
    $\Omega:=dd_{J}E$, called the fundamental form,  has a maximal rank.
     \end{description}
 \end{defn}
 From now on, we  will be placed on a Finsler space $(M,E)$.
\begin{thm}\label{spray} Let  $(M,E)$ be a  Finsler space. The vector field $S$ defined by $i_{S}\Omega =-dE$ is a spray, called the canonical spray.
 \end{thm}
\begin{thm}\label{Barthel} On a Finsler space $(M,E)$, there exists a unique conservative \emph{(}$d_hE=~0$\emph{)} homogeneous nonlinear  connection $\Gamma$ with zero torsion.  It is given by $\Gamma = [J,S]$, where $S$ is the canonical spray.
Such a connection is called the canonical or Barthel connection  associated with $(M,E)$.
\end{thm}
%%%%%%%%%%%%%%%%%%%%%%%%%%%%%%%%%%%%%%%%%%%%%%%%%%%%%%%%%%%%%%%%%%%%%%%%%&&&&&&&&&&&&&&&&&&&&&&&&&&&&&&&&&&&&&&&&&&&&&&&&&&&&&&&&&&&&&&&&
\begin{thm}  For a Finsler space  $(M,E)$,
there exists a unique linear connection  \, $\overcirc{D}$ on $TM$  satisfying the following properties:
\begin{description}
                  \item[(a)]$\overcirc{D}J=0$,\hspace{5.4cm}\em{\textbf{(b)}}\, $\overcirc{D}C=v$,
                  \item[(c)] $\overcirc{D}\Gamma=0\,\,
                  (\Longleftrightarrow\overcirc{D}h=\overcirc{D}v=0  )$, \hspace{1.7cm}{\textbf{(d)}}\, $\overcirc{D}_{JX}JY=J[JX,Y]$,
                  \item[(e)]$\overcirc{T}(JX,Y)=0$,
 \end{description}
 $h$ and $v$ being the horizontal and vertical projectors of $\Gamma=[J,S]$ and \, $\overcirc{T}$ is the (classicl) torsion of \, $\overcirc{D}$.
 This connection is called the Berwald connection.
\end{thm}
Berwald connection is completely  determined  by:
\begin{equation}\label{berwaldconn.}
     \overcirc{D}_{JX}JY=J[JX,Y],\quad
\overcirc{D}_{hX}JY=v[hX,JY],\quad
  \overcirc{D}F=0.
\end{equation}
The (h)h-torsion  \, $\overcirc{T}(hX,hY)$ of Berwald connection is given by \,
$\overcirc{T}(hX,hY)=\mathfrak{R}(X,Y).$

\vspace{7pt}

A metric $g$ can be defined on $TM$ by
\begin{equation}\label{metricg}
 g(X,Y)=\Omega(X,FY), \,\,\forall X,Y\in \mathfrak{X}(TM),
\end{equation}
where $F$ is the almost-complex structure associated with Barthel connection.
\begin{thm} For a Finsler space $(M,E)$, there exists a unique linear connection  ${D}$ on $TM$ satisfying the following properties:
\begin{description}
                  \item[(a)]${D}J=0$,\hspace{4.4cm} \em{\textbf{(b)}} ${D}C=v$,
                  \item[(c)] ${D}\Gamma=0$,\hspace{4.6cm}\textbf{(d)} ${D}g=0$,
                  \item[(e)] ${T}(JX,JY)=0$,\hspace{3.25cm}\textbf{(f)} $JT(hX,hY)=0$.
 \end{description}
 This connection  is called the Cartan  connection.
 \end{thm}
Cartan  connection is completely determined by:
\begin{equation}\label{cartanconn.}
    D_{JX}JY=\overcirc{D}_{JX}JY+\mathcal{C}(X,Y),\quad
  D_{hX}JY=\overcirc{D}_{hX}JY+\mathcal{C}'(X,Y),\quad
  {D}F=0,
\end{equation}
where $\mathcal{C}$ and $\mathcal{C}'$ are the vector 2-forms on $TM$ defined respectively~by
{\vspace{-10pt}}$$\Omega(\mathcal{C}(X,Y),Z)=\frac{1}{2}(\mathcal{L}_{JX}(J^\ast g))(Y,Z),\quad
\Omega(\mathcal{C}'(X,Y),Z)=\frac{1}{2}(\mathcal{L}_{hX}g)(JY,JZ),$$
The tensors  ${\mathcal{C}}$ and $\mathcal{C}'$   are symmetric, semi-basics   and
  \begin{equation}\label{c(s)}
  {\mathcal{C}}(X,S)=\mathcal{C}'(X,S)=0.
  \end{equation}
The $h$-torsion,  $hv$-torsion and $v$-torsion of Cartan connection are  given respectively by:
\begin{equation*}\label{cart.tors.}
{T}(hX,hY)=\mathfrak{R}(X,Y),\quad T(hX,JY)=(\mathcal{C}'-F\mathcal{C})(X,Y),\quad T(JX,JY)=0.
\end{equation*}

Let\, $\overcirc R$ and\, $\overcirc P$ be the h- and hv-curvature tensors of Berwald connection and let $R$, $P$ and $Q$ be the h-, hv- and v-curvature tensors of Cartan connection.
\begin{lem}\label{car.curv.}
For Cartan connection,  the following properties hold:
\begin{description}
  \item[(a)]$R(X,Y)S=\mathfrak{R}(X,Y)$.\hspace{1.2cm}
  \textbf{\emph{(b)}} $P(X,Y)S=\mathcal{C}'(X,Y)$.
  \item[(c)]$P(S,X)Y=P(X,S)Y=0$.
\end{description}
\end{lem}
\vspace{-10pt}
%%%%%%%%%%%%%%%%%%%%%%%%%%%%%%%%%%%%%%%%%%%%%%%%%%%%%%%%%%%%%%%%%%%%%%%%%%%%%%%%
\Section{Chern connection}
In this section, we prove the existence and uniqueness theorem of a remarkable connection: the Chern connection. We also give an   explicit expressions for Chern connection.

\vspace{5pt}
We begin with some definitions quoted from \cite{r22}.
\vspace{-5pt}
\begin{defn}
A linear connection $\nabla$ on $TM$ is said to be regular if $\nabla J=0$ and the map $\varphi:V(TM)\longrightarrow V(TM)$:
  $X\longmapsto \nabla_XC$  is an isomorphism of $V(TM)$.
\end{defn}
 For a regular connection $\nabla$ on $TM$ there is associated a nonlinear   connection $\Gamma$  on $M$ defied by $\Gamma={I-2\varphi^{-1}}\circ
 \nabla {C}$;\, $\Gamma$ is said to be induced by $\nabla$.

\begin{defn}
A regular  connection $\nabla$ on $TM$    is
 said to be reducible if ${\nabla\Gamma}=0$, where $\Gamma$ is the nonlinear connection induced  by $\nabla$.
\end{defn}

\begin{defn}
A linear  connection $\nabla$ on $TM$ is said to be almost-projectable  if $\nabla{J}=0$ and $\nabla_{JX}C=JX$, for all $X\in T(TM)$. (an almost-projectable connection is necessairly regular).

If we replace the last axiom  by the axiom  $\nabla_{JX}JY=J[JX,Y]$,  the connection $\nabla$ is called normal almost-projectable.

 The connection $\Gamma$ induced on $M$ by an almost-projectable (rep. normal almost-projectable) connection $\nabla$ on $TM$  will be called a projection (rep. normal projection) of $\nabla$. We also say that $\nabla$ projects (resp. projects normally) on $\Gamma$.
\end{defn}

\begin{defn}
Let $\Gamma$ be a connection on    $M$.  The lift of $\Gamma$ is a reducible  connection $\nabla$ on $TM$ which projects
on $\Gamma$. The lift of $\Gamma$ is said to be normal if $\nabla$ is normal.
\end{defn}

\begin{lem}\label{df=0}
For a reducible connection $\nabla$, we have ${\nabla F}=0$, where ${F}$ is the almost-complex structure associated with
 the connection $\Gamma$ induced by $\nabla$.
\end{lem}

Now, we are in a position to announce our fundamental result.
\vspace{-8pt}
\begin{thm}\label{chernconnc.} For a Finsler manifold $(M,E)$ there exists a unique normal
lift  \, $\overast{D}$ of Barthel connection $\Gamma=[J,S]$ such that:
\begin{description}
  \item[(a)] $\overast{D}$ is horizontally metric:\, $\overast{D}_{hX}g=0,\,\,\forall X\in \cppp$.
  \item[(b)] The classical torsion \,\,$\overast{T}$ has the property that: $J\,\,\overast{T}(hX,hY)=0, \, \forall X,Y\in \cppp$.
\end{description}
This connection  is called Chern connection.
\end{thm}

\begin{proof} Firstly,  we prove the \textbf{uniqueness}.
Since \,   $\overast{D}$ is a normal lift of  Barthel connection $\Gamma=[J,S]$, then
\begin{equation}\label{naplajj}
\overast{D}_{JX}JY=J[JX,Y].
\end{equation}
 Also,  by Lemma \ref{df=0}, we have
 \begin{equation}\label{naplaf=0}
 \overast{D} F=0.
 \end{equation}
 Condition \textbf{(a)} implies that:
\begin{eqnarray}
% \nonumber to remove numbering (before each equation)
\label{1} hX.g(JY,JZ)  &=&g(\,\overast{D}_{hX}JY,JZ)+g(JY,\,\overast{D}_{hX}JZ),  \\
\label{2}   hY.g(JZ,JX)  &=&g(\,\overast{D}_{hY}JZ,JX)+g(JZ,\,\overast{D}_{hY}JX), \\
\label{3}   hZ.g(JX,JY)  &=&g(\,\overast{D}_{hZ}JX,JY)+g(JX,\,\overast{D}_{hZ}JY).
\end{eqnarray}
By adding  (\ref{1}), (\ref{2}) and subtracting  (\ref{3}), we get
\begin{eqnarray}\label{4}
% \nonumber to remove numbering (before each equation)
\nonumber hX.g(JY,JZ)&+&hY.g(JZ,JX) - hY.g(JZ,JX)= g(\,\overast{D}_{hX}JY+\,\overast{D}_{hY}JX,JZ)  \\
   && +g(JY,\,\overast{D}_{hX}JZ-\,\overast{D}_{hZ}JX)+g(\,\overast{D}_{hY}JZ-\,\overast{D}_{hZ}JY,JX).
\end{eqnarray}
Condition \textbf{(b)} together with \,$\overast{D}J=0$ imply:
\begin{equation}\label{jt}
\overast{D}_{hX}JY-\overast{D}_{hY}JX=J[hX,hY].
\end{equation}
From  (\ref{4}) and (\ref{jt}), we get
\begin{eqnarray}\label{gnabla}
% \nonumber to remove numbering (before each equation)
\nonumber  g(2\,\overast{D}_{hX}JY,JZ)  &=&hX.g(JY,JZ)+hY.g(JZ,JX) - hY.g(JZ,JX)\\
&&{\hspace{-2cm}}+g(J[hX,hY],JZ)- g(J[hX,hZ],JY)-g(J[hY,hZ],JX)
\end{eqnarray}
Since $\Omega(X,Y)=g(X,JY)-g(JX,Y)$ and using $J{\mathcal{C}'}=0$, then $$\frac{1}{2}(\mathcal{L}_{hX}g)(JY,JZ)=\Omega(\mathcal{C}'(X,Y),Z)=g(\mathcal{C}'(X,Y),JZ),$$ which is totally symmetric. Now,
\begin{eqnarray*}
% \nonumber to remove numbering (before each equation)
  g(2\mathcal{C}'(X,Y),JZ) &=& hX.g(JY,JZ)-g([hX,JY],JZ)-g(JY,[hX,JZ]),\\
  g(2\mathcal{C}'(Y,Z),JX) &=& hY.g(JZ,JX)-g([hY,JZ],JX)-g(JZ,[hY,JX]),\\
  -g(2\mathcal{C}'(Z,X),JY) &=&- hZ.g(JX,JY)+g([hZ,JX],JY)+g(JX,[hZ,JY]).
\end{eqnarray*}
By adding the above three equations, we get
\begin{eqnarray}\label{cdash}
  \nonumber  g(2\mathcal{C}'(X,Y),JZ)&=&hX.g(JY,JZ)+hY.g(JZ,JX)-hZ.g(JX,JY)\\
 \nonumber  && -g([hX,JY]+[hY,JX],JZ)+g([hZ,JX]-[hX,JZ],JY)\\&&
    +g([hZ,JY]-[hY,JZ],JX)
\end{eqnarray}
From  (\ref{gnabla}) and (\ref{cdash}), we have
\begin{eqnarray}\label{gnapla}
% \nonumber to remove numbering (before each equation)
 \nonumber  g(2\,\overast{D}_{hX}JY,JZ)&=&g(2\mathcal{C}'(X,Y),JZ)+ g([hX,JY]+[hY,JX]+J[hX,hY],JZ)\\
 \nonumber &&- g([hZ,JX]-[hX,JZ]+J[hX,hZ],JY) \\
   &&-g([hZ,JY]-[hY,JZ]+J[hY,hZ],JX).
\end{eqnarray}
Since the Barthel connection is torsion-free (Theorem \ref{Barthel}), then $$0=t(X,Y)=v[JX,hY]+v[hX,JY]-J[hX,hY]$$ and so $J[hX,hY]=v[JX,hY]+v[hX,JY]$.
Hence, one can write:
\begin{eqnarray*}
% \nonumber to remove numbering (before each equation)
   [hX,JY]+[hY,JX]+J[hX,hY]&=&2v[hX,hY]+h[hX,JY]+h[hY,JX],\\
\textcolor[rgb]{1.00,1.00,1.00}{(}[hZ,JX]-[hX,JZ]+J[hX,hZ]&=&h[JZ,JX]+h[hZ,JX],\\
\textcolor[rgb]{1.00,1.00,1.00}{(}[hZ,JY]-[hY,JZ]+J[hY,hZ] &=&h[hZ,JY]+h[JZ,hY].
\end{eqnarray*}
The above relations and the the fact that $g(hX,JY)=0$ enable us to write  (\ref{gnapla}) in the form
$$g(2\,\overast{D}_{hX}JY,JZ)=g(2\mathcal{C}'(X,Y)+2v[hX,JY],JZ).$$
Hence,
\begin{equation}\label{naplahj}
\overast{D}_{hX}JY=v[hX,JY]+\mathcal{C}'(X,Y).
\end{equation}

Therefore, \, $\overast{D}_XY$ is uniquely determined by
(\ref{naplajj}), (\ref{naplahj}) and (\ref{naplaf=0}).\\

To prove  the \textbf{existence} of \,$\overast{D}$, let us  define \,$\overast{D}$    by  the requirement that  (\ref{naplajj}), (\ref{naplahj}) and (\ref{naplaf=0}) hold.
Now, we have to prove that\,  $\overast{D}$ is a normal lift of $\Gamma=[J,S]$ (i.e.,    $\,\overast{D} J=0$, \, $\,\overast{D} C=v$,
\, $\,\overast{D} \Gamma=0$, \, $\,\overast{D}_{JX}JY=J[JX,Y]$) and conditions \textbf{(a)} and \textbf{(b)} are satisfied.

\vspace{5pt}
\noindent $\bullet$ $\,\overast{D} J=0$:   From (\ref{naplajj}),
(\ref{naplaf=0}) and
 (\ref{naplahj}), we have
\begin{eqnarray*}
% \nonumber to remove numbering (before each equation)
   J\,\overast{D}_{hX}Y&=& J\,\overast{D}_{hX}hY+J\,\overast{D}_{hX}vY \\
   &=&JF\,\overast{D}_{hX}JY+J\,\overast{D}_{hX}JFY  \\
   &=&(JFv[hX,JY]+v \mathcal{C}'(X,Y))+(Jv[hX,vY]+J\mathcal{C}'(X,FY))\\
   &=&v[hX,JY]+ \mathcal{C}'(X,Y),\,\, \text{since $\mathcal{C}'$ is semi basic and $Jv=0$ }\\
   &=&\overast{D}_{hX}JY.
\end{eqnarray*}
Similarly, one can show that $J\,\overast{D}_{vX}Y=\,\overast{D}_{vX}JY$.

\vspace{5pt}
\noindent$\bullet$ $\,\overast{D} C=v$:    From
  (\ref{naplahj}), (\ref{c(s)})   and  (\ref{chx}), we get
$$  \,\overast{D}_{hX}C =\,\overast{D}_{hX}JS= v[hX,JS]+ \mathcal{C}'(X,S)
   =-v[C,hX]
   =-vh[C,X]
   =0.$$
On the other hand,   from (\ref{naplajj}) and  (\ref{jx}), we obtain
$  \,\overast{D}_{JX} C =\,\overast{D}_{JX}JS= J[JX,S]=JX.$

\vspace{5pt}

\noindent$\bullet$ $\,\overast{D} \Gamma=0$ or,  equivalently,  $\,\overast{D} h=0$:
\begin{eqnarray*}
% \nonumber to remove numbering (before each equation)
   h\,\overast{D}_{hX}Y&=& h\,\overast{D}_{hX}hY+h\,\overast{D}_{hX}vY=h\,\overast{D}_{hX}hY, \,\, \text{since $\,\overast{D}_{hX}vY$ is vertical by (\ref{naplahj})} \\
   &=&h\,\overast{D}_{hX}FJY=hF\,\overast{D}_{hX}JY ,\,\, \text{by (\ref{naplaf=0})} \\
   &=& hFv[hX,JY]+hF \mathcal{C}'(X,Y)\\
   &=&Fv^2[hX,JY]+ Fv\mathcal{C}'(X,Y)\\
   &=&Fv[hX,JY]+ F\mathcal{C}'(X,Y)=F{D}_{hX}JY
    =\overast{D}_{hX}hY.
\end{eqnarray*}
Similarly, \, $\overast{D}_{JX}hY=h \, \overast{D}_{JX}Y$.

\vspace{5pt}

\noindent$\bullet$ $\overast{D}$ is h-metrical: As $g(JX,JY)=g(hX,hY)$ and $g(hX,JY)=~0$, it suffices to prove that
$\,(\overast{D}_{hX}g)(JY,JZ)=0$. By  (\ref{naplajj}),  (\ref{naplaf=0}) and (\ref{naplahj}), we have
\begin{eqnarray*}
% \nonumber to remove numbering (before each equation)
 (\,\overast{D}_{hX}g)(JY,JZ)  &=& hX.g(JY,JZ)- g(\,\overast{D}_{hX}JY,JZ)-g(JY,\,\overast{D}_{hX}JZ) \\
   &=&  hX.g(JY,JZ)- g(v[hX,JY],JZ)-g(\mathcal{C}'(X,Y),JZ)\\
   &&-g(JY,v[hX,JZ])-g(JY,\mathcal{C}'(X,Z))\\
   &=&hX.g(JY,JZ)- g(v[hX,JY],JZ)-g(JY,v[hX,JZ])  \\
   &&-2\mathcal{C}'_b(X,Y,Z)=0.
\end{eqnarray*}

\vspace{3pt}

\noindent$\bullet$ $J\textbf{T}(hX,hY)=0$:  By   (\ref{naplahj}) and the symmetry of $\mathcal{C}'$ , we have
\begin{eqnarray*}
% \nonumber to remove numbering (before each equation)
  J\,\overast{T}(hX,hY) &=&J\,\overast{D}_{hX}hY-J\,\overast{D}_{hY}hX-J[hX,hY]  \\
   &=& \,\overast{D}_{hX}JY-\,\overast{D}_{hY}JX-J[hX,hY] \\
   &=&v[hX,JY]+ \mathcal{C}'(X,Y)-v[hY,JX]- \mathcal{C}'(Y,X)-J[hX,hY]\\
   &=& t(X,Y)=0.
\end{eqnarray*}
This completes the proof.
\end{proof}

\begin{cor}\label{chern.nab.} The Chern connection $\,\overast{D}$ is completely determined by:
\begin{description}
  \item[(a)] $\,\overast{D}_{JX}JY=J[JX,Y]=\,\overcirc{D}_{JX}JY.$
  \item[(b)] $\,\overast{D}_{hX}JY=v[hX,JY]+\mathcal{C}'(X,Y)={D}_{hX}JY.$
  \item[(c)] $\,\overast{D} F=0.$
\end{description}
\end{cor}
\vspace{-10pt}
%%%%%%%%%%%%%%%%%%%%%%%%%%%%%%%%%%%%%%%%%%%%%%%%%%%%%%%%%%%%%%%%%%%%%%%%%%%%%%%%%%%%%%%%
\Section{ Torsion and curvature tensors }

In this section, we study the torsion and curvature tensors of Chern connection.  We also derive  the Bianchi identities and obtain some properties of the curvature tensors. We start with the following lemma which will be useful for subsequent use.

\begin{lem}\label{chern.[]}For all $X,Y\in \cppp$, we have
\begin{description}
  \item[(a)] $[JX,JY]=J(\,\overast{D}_{JX}Y-\,\overast{D}_{JY}X).$
  \item[(b)] $[hX,JY]=J(\,\overast{D}_{hX})Y-h(\,\overast{D}_{JY}X)-{\mathcal{C}'}(X,Y).$
  \item[(c)] $[hX,hY]=h(\,\overast{D}_{hX}Y-\,\overast{D}_{hY}X)-\mathfrak{R}(X,Y).$
\end{description}
\end{lem}

\begin{proof}~\par
\noindent \textbf{(a)} By Corollary \ref{chern.nab.} and the fact that $[J,J]=0$ (\ref{J}), we get
$$   J(\,\overast{D}_{JX}Y-\,\overast{D}_{JY}X)=\,\overast{D}_{JX}JY-\,\overast{D}_{JY}JX
   =J[JX,Y]-J[JY,X]
   =[JX,JY].$$
\noindent \textbf{(b)} By Corollary \ref{chern.nab.} and the identity  $h[JY,X]=h[JY,hX]$, we obtain
\begin{eqnarray*}
 % \nonumber to remove numbering (before each equation)
   J(\,\overast{D}_{hX}Y)-h(\,\overast{D}_{JY}X)&=&\,\overast{D}_{hX}JY-\,\overast{D}_{JY}hX  \\
   &=&v[hX,JY]+{\mathcal{C}'}(X,Y)-h[JY,X] \\
   &=&v[hX,JY]-h[JY,X]+{\mathcal{C}'}(X,Y)\\
   &=&v[hX,JY]+h[hX,JY]+{\mathcal{C}'}(X,Y)\\
   &=&[hX,JY]+\mathcal{C}'(X,Y).
 \end{eqnarray*}
 \noindent \textbf{(c)} Again by Corollary \ref{chern.nab.} and   the symmetry property of $\mathcal{C}'$, we have
\begin{eqnarray*}
 % \nonumber to remove numbering (before each equation)
   h(\,\overast{D}_{hX}Y-\,\overast{D}_{hY}X)&=&\,\overast{D}_{hX}hY-\,\overast{D}_{hY}hX  \\
   &=&Fv[hX,JY]+F\mathcal{C}'(X,Y)-Fv[hY,JX]-F\mathcal{C}'(Y,X) \\
      &=&Fv[hX,JY]+Fv[JX,hY].
    \end{eqnarray*}
 As the torsion of $\Gamma$  vanishes, then $0=t(X,Y)=v[JX,hY]+v[hX,JY]-J[hX,hY]$, from which  $Fv[JX,hY]+Fv[hX,JY]=FJ[hX,hY]=h[hX,hY]$. Consequently,
 $$ h(\,\overast{D}_{hX}Y-\,\overast{D}_{hY}X) =h[hX,hY]=[hX,hY]-v[hX,hY]=
   [hX,hY]+\mathfrak{R}(X,Y),$$
  where we have used the identity $\mathfrak{R}(X,Y)=-v[hX,hY]$ \cite{Nabil.2}.
\end{proof}
\begin{rem}\em{The last  identity of Lemma \ref{chern.[]} retrieves a result of \cite{Nabil.2}: A necessary and sufficient condition for the horizontal
 distribution to be completely integrable is that  $\mathfrak{R}$ vanishes}.
\end{rem}
\begin{prop}
The $h$-torsion, $hv$-torsion  and $v$-torsion of Chern connection $\,\overast{D}$ are given by:
\begin{description}
  \item[(a)]$\,\overast{T}(hX,hY)=\mathfrak{R}(X,Y).$
  \item[(b)]$\,\overast{T}(hX,JY)=\mathcal{C}'(X,Y).$
   \item[(c)]$\,\overast{T}(JX,JY)=0.$
\end{description}
\end{prop}
\begin{proof}
~\par
\noindent \textbf{(a)} Follows directly from the definition of $\,\overast{T}(hX,hY)$ and Lemma \ref{chern.[]} \textbf{(c)}.

\noindent \textbf{(b)} By Corollary \ref{chern.nab.} and using  the property  that $h[JX,vY]=0$, we get
\begin{eqnarray*}
% \nonumber to remove numbering (before each equation)
 \phantom{ggggg} \overast{T}(hX,JY) &=& \,\overast{D}_{hX}JY-\,\overast{D}_{JY}hX-[hX,JY] \\
   &=& v[hX,JY]+\mathcal{C}'(X,Y)- h[JY,X]-[hX,JY]\\
   &=& v[hX,JY]+\mathcal{C}'(X,Y)- h[JY,hX]-(h[hX,JY]+v[hX,JY])\\
    &=& \mathcal{C}'(X,Y).%\hspace{9cm}\qed
    \end{eqnarray*}

\noindent \textbf{(c)}  Is obvious.
\end{proof}

As \,  $\overast{D}F=0$, the (classical) curvature tensor $K$ of Chern connection is completely determined by the three curvature tensors:  $h$-curvature  \, $\overast{R}$,  $hv$-curvature \, $\overast{P}$ and $v$-curvature \, $\overast{Q}$ defined respectively by:
\begin{eqnarray*}
% \nonumber to remove numbering (before each equation)
   \overast{R}(X,Y)Z&=& K(hX,hY)JZ, \\
   \overast{P}(X,Y)Z&=& K(hX,JY)JZ,\\
  \overast{Q}(X,Y)Z&=&  K(JX,JYJ)Z.
\end{eqnarray*}
\begin{prop}\label{cherncurv.}
The h-curvature \, $\overast{R}$,  hv-curvature \,  $\overast{P}$  and  v-curvature \, $\overast{Q}$ of the Chern connection are given by:
\begin{description}
  \item[(a)] $\overast{R}(X,Y)Z=R(X,Y)Z-\mathcal{C}(F\mathfrak{R}(X,Y),Z).$
  \item[(b)] $\overast{P}(X,Y)Z=\overcirc{P}(X,Y)Z-(\,\,\overast{D}_{JY}\mathcal{C}')(X,Z).$
  \item[(c)] $\overast{Q}(X,Y)Z=0$.
\end{description}
\end{prop}
\begin{prof} We prove \textbf{(a)} only. The other expressions can be proved similarly.
As $\,\overast{D}_{hX}JY=D_{hX}JY$ (Corollary \ref{chern.nab.}(b)), we have
\begin{eqnarray*}
% \nonumber to remove numbering (before each equation)
  \overast{R}(X,Y)Z
   &=&\,\overast{D}_{hX}\,\overast{D}_{hY}JZ-\,\overast{D}_{hY}\,\overast{D}_{hX}JZ-\,\overast{D}_{[hX,hY]}JZ \\
   &=& D_{hX}D_{hY}JZ-D_{hY}D_{hX}JZ-\,\overast{D}_{[hX,hY]}JZ \\
   &=& R(X,Y)Z+D_{[hX,hY]}JZ-\,\overast{D}_{[hX,hY]}JZ \\
   &=& R(X,Y)Z+D_{JF[hX,hY]}JZ-\,\overast{D}_{JF[hX,hY]}JZ
\end{eqnarray*}
By (\ref{cartanconn.}) and Corollary \ref{chern.nab.}, the last equation takes the form
\begin{eqnarray*}
\overast{R}(X,Y)Z
   &=&R(X,Y)Z+\mathcal{C}(F[hX,hY],Z)\\
   &=&R(X,Y)Z-\mathcal{C}(F\mathfrak{R}(X,Y),Z),
\end{eqnarray*}
where we have used the identity $\mathfrak{R}(X,Y)=-v[hX,hY]$ and the fact that $C$ is semi-basic.
\end{prof}

\begin{prop}\label{R,P,S chern}
The $h$-curvature \, $\overast{R}$ and  $hv$-curvature \, $\overast{P}$  of  Chern connection have the following properties:
\begin{description}
  \item[(a)] $\overast{R}(X,Y)S=\mathfrak{R}(X,Y).$
  \item[(b)] $\overast{P}(X,Y)S=\overast{P}(S,Y)X=\mathcal{C}'(X,Y).$
  \item[(c)] $\overast{P}(X,S)Z=0$.
\end{description}
\end{prop}
\begin{prof}
~\par
\noindent \textbf{(a)} Follows from Proposition \ref{cherncurv.}, Lemma \ref{car.curv.} and (\ref{c(s)}).

\noindent \textbf{(b)} By Proposition \ref{cherncurv.},  (\ref{jx}), the properties of $\mathcal{C}'$ and the
 properties of \, $\overcirc{P}$, we get
\vspace{-6pt}
\begin{eqnarray*}
 \overast{P}(X,Y)S &=&-(\,\overast{D}_{JY}\mathcal{C}')(X,S)=- (\,\overcirc{D}_{JY}\mathcal{C}')(X,S)=\mathcal{C}'(X,\,\overcirc{D}_{JY}S)=\mathcal{C}'(X,FJ[JY,S])\\
&=&\mathcal{C}'(X,FJY)
                   =\mathcal{C}'(X,Y).
\vspace{-6pt}
\end{eqnarray*}

\noindent \textbf{(c)} can be proved similarly.
\end{prof}

\vspace{8pt}
To study the Bianchi identities for Chern connection, let us first write  the Bianchi identities for an arbitrary connection $\nabla$.
\begin{lem}\label{generalbianchi} Let $\nabla$ be a linear connection on $M$ with  torsion  tensor $\textbf{T}$ and curvature tensor $\textbf{K}$. For every
$X,Y,Z \in \mathfrak{X}(M)$, we
have\,\emph{:}\,
\begin{description}
     \item[(I)]$\mathfrak{S}_{X,Y,Z}\{\textbf{K}(X,Y) Z\}=\mathfrak{S}_{X,Y,Z}\{{\textbf{T}(\textbf{T}}(X,Y),Z)
   +(\nabla_X{T})(Y,Z) \}$,

\item[(II)]
$\mathfrak{S}_{X,Y,Z}\{{\textbf{K}}({\textbf{T}}(X,Y),Z)+(\nabla_X\textbf{K})(Y,Z)\}=0$,
\end{description}
where the symbol $\mathfrak{S}_{X,Y,Z}$ denotes cyclic sum over
$X$,$Y$ and $Z$.
\end{lem}

Applying the identities (I) and (II) on Chern connection, for different triples of vector fields $(hX,hY,hZ), (hX,hY,JZ)$, ..., we obtain many identities. Here, we give only the most important of these identities.
\begin{prop}\label{bianchichern}
The first Bianchi identity for Chern connection yields:
\begin{description}
  \item[(a)]$\mathfrak{S}_{X,Y,Z}\, \{\,\overast{R}(X,Y)Z\}=0$.

  \item[(b)] $\mathfrak{S}_{X,Y,Z}\,\{(\,\overast{D}_{hX}\,\mathfrak{R})(Y,Z)\}=\mathfrak{S}_{X,Y,Z}\{\,
  \mathcal{C}'(F\mathfrak{R}(X,Y),Z)\}$.

  \item[(c)] $\,\overast{P}(X,Y)Z=\,\overast{P}(Z,Y)X$.

  \item[(d)] $\,\overast{P}(X,Y)Z-\overast{P}(X,Z)Y=(\,\overast{D}_{JZ}\mathcal{C}')(X,Y)-(\,\overast{D}_{JY}\mathcal{C}')(X,Z)$.

\end{description}
The second Bianchi identity for Chern connection yields:
\begin{description}
  \item[(e)]$\mathfrak{S}_{X,Y,Z}\{\,(\,\overast{D}_{hX}\,\overast{R})(Y,Z)\}=\mathfrak{S}_{X,Y,Z}\{\, \overast{P}(X,F\mathfrak{R}(Y,Z))\}$.

  \item[(f)]$(\,\overast{D}_{hX}\,\overast{P})(Y,Z)-(\,\overast{D}_{hY}\,\overast{P})(X,Z)+(\,\overast{D}_{JZ}\,\overast{R})(X,Y)
  =\overast{P}(X,F\mathcal{C}'(Y,Z))\\
  -\overast{P}(Y,F\mathcal{C}'(X,Z))$.

  \item[(g)]$(\,\overast{D}_{JY}\,\overast{P})(X,Z)=(\,\overast{D}_{JZ}\,\overast{P})(X,Y)$.
\end{description}
\end{prop}

\begin{cor}\label{nablacr} The h-curvature\,  $\overast{R}$ and the hv-curvature\, $\overast{P}$ satisfy:\\
\textbf{\em(a)} $\overast{D}_C\,\,\overast{R}=0$, \quad \textbf{\em(b)}\, $\overast{D}_C\,\,\overast{P}=0$, \quad \textbf{\em(c)}\, $\overast{P}$ is totally symmetric if \, $\overast{D}_{JZ}\,\mathcal{C}'=0$.
\end{cor}

\begin{prop}\label{chenh-curv.}
The h-curvature \, $\overast{R}$ has the following properties:
\begin{description}
  \item[(a)] $\overast{R}(X,Y,Z,W)=-\,\overast{R}(Y,X,Z,W)$,
  \item[(b)] $\overast{R}(X,Y,Z,W)+\,\overast{R}(Y,Z,X,W)+\,\overast{R}(Z,X,Y,W)=0$,
\end{description}
\end{prop}
\noindent Moreover, if $\mathfrak{R}=0$, we have
\begin{description}
  \item[(c)] $\overast{R}(X,Y,Z,W)=-\,\overast{R}(X,Y,W,Z)$,
  \item[(d)]  $\overast{R}(X,Y,Z,W)=\,\overast{R}(Z,W,X,Y)$,
\end{description}
where \, $\overast{R}(X,Y,Z,W):=g(\,\overast{R}(X,Y)Z,JW)$.
\begin{proof}
{\textbf{(a)}} is clear.

\noindent {\textbf{(b)}} We have\\
${\hspace{2cm}}\overast{R}(X,Y,Z,W)+\,\overast{R}(Y,Z,X,W)+\,\overast{R}(Z,X,Y,W)\\
{\hspace{2.8cm}}=g(\,\overast{R}(X,Y)Z,JW)+g(\,\overast{R}(Y,Z)X,JW)+g(\,\overast{R}(Z,X)Y,JW)\\
{\hspace{2.8cm}}=g(\,\overast{R}(X,Y)Z+\,\overast{R}(Y,Z)X+\,\overast{R}(Z,X)Y,JW)=0$, by Proposition \ref{bianchichern}.

\noindent {\textbf{(c)}} By Theorem \ref{chernconnc.}, we have
$$hX.g(JY,JZ)= g(\,\overast{D}_{hX}JY,JZ)+g(JY,\,\overast{D}_{hX}JZ). $$
Then, we can write\\
${\hspace{2cm}}hW.(hX.g(JY,JZ))=g(\,\overast{D}_{hW}\,\overast{D}_{hX}JY,JZ)+g(\,\overast{D}_{hX}JY,\,\overast{D}_{hW}JZ)\\
{\hspace{4cm}}+g(\,\overast{D}_{hW}JY,\,\overast{D}_{hX}JZ)+g(JY,\,\overast{D}_{hW}\,\overast{D}_{hX}JZ).$ \\
Interchanging  $X$ and $W$, we get\\
${\hspace{2cm}}hX.(hW.g(JY,JZ))=g(\,\overast{D}_{hX}\,\overast{D}_{hW}JY,JZ)+g(\,\overast{D}_{hW}JY,\,\overast{D}_{hX}JZ)\\
{\hspace{4cm}}+g(\,\overast{D}_{hX}JY,\,\overast{D}_{hW}JZ)+g(JY,(\,\overast{D}_{hX}\,\overast{D}_{hW}JZ)).$\\
Using the above two equations, we obtain \\
${\hspace{2cm}}[hW,hX].g(JY,JZ)=g((\,\overast{D}_{hW}\,\overast{D}_{hX}-\,\overast{D}_{hX}\,\overast{D}_{hW})JY,JZ)\\
{\hspace{6cm}}+g(JY,(\,\overast{D}_{hW}\,\overast{D}_{hX}-\,\overast{D}_{hX}\,\overast{D}_{hW})JZ).$\\
If  $\mathfrak{R}=0$, then the horizontal distribution is completely integrable. Consequently, $[hW,hX]$ is horizontal and so $\,\overast{D}_{[hW,hX]}\,g=0$. Hence, we have\\
${\hspace{2cm}}[hW,hX].g(JY,JZ)=g(\,\overast{D}_{[hW,hX]}JY,JZ)+g(JY,\,\overast{D}_{[hW,hX]}JZ).$\\
Comparing the above two equations, we get\\
${\hspace{3cm}}g((\,\overast{D}_{hW}\,\overast{D}_{hX}-\,\overast{D}_{hX}\,\overast{D}_{hW}-\,\overast{D}_{[hW,hX]})JY,JZ)\\
{\hspace{3cm}}+g(JY,(\,\overast{D}_{hW}\,\overast{D}_{hX}-\,\overast{D}_{hX}\,\overast{D}_{hW}-\,\overast{D}_{[hW,hX]})JZ)=0.$\\
From which, $$\overast{R}(W,X,Y,Z)=-\,\overast{R}(W,X,Z,Y).$$

\noindent {\textbf{(d)}} Using \textbf{(a)}, \textbf{(b)} and \textbf{(c)}, since $\mathfrak{R}=0$, we have
$$  \overast{R}(X,Y,Z,W) =-\,\overast{R}(Y,X,Z,W)
   =  \,\overast{R}(X,Z,Y,W)+\,\overast{R}(Z,Y,X,W),$$
$$ \overast{R}(X,Y,Z,W) =-\,\overast{R}(X,Y,W,Z)
  = \,\overast{R}(Y,W,X,Z)+\,\overast{R}(W,X,Y,Z).$$
Adding the above two equation, we get
\begin{eqnarray}\label{eq.1}
% \nonumber to remove numbering (before each equation)
\nonumber  2\,\overast{R}(X,Y,Z,W) &=&\,\overast{R}(X,Z,Y,W)+\,\overast{R}(Z,Y,X,W)  \\
  &&  +\,\overast{R}(Y,W,X,Z)+\,\overast{R}(W,X,Y,Z).
\end{eqnarray}
Similarly,
\begin{eqnarray}\label{eq.2}
% \nonumber to remove numbering (before each equation)
\nonumber  2\,\overast{R}(Z,W,X,Y) &=&\,\overast{R}(Z,X,W,Y)+\,\overast{R}(X,W,Z,Y)  \\
  &&  +\,\overast{R}(W,Y,Z,X)+\,\overast{R}(Y,Z,W,X).
\end{eqnarray}
Comparing  (\ref{eq.1}) and (\ref{eq.2}), we get \, $\overast{R}(X,Y,Z,W)=\,\overast{R}(Z,W,X,Y)$.
\end{proof}
\begin{rem}\em{It is to be noted that if\, $\mathfrak{R}$ vanishes, we get some interesting results:
\begin{description}
  \item[$\bullet$] $\,\overast{R}=R$ (Proposition \ref{cherncurv.} \textbf{(a)}) and also $\mathfrak{S}_{X,Y,Z}\, \{{R}(X,Y)Z\}=0$.
  \item[$\bullet$] $\mathfrak{S}_{X,Y,Z}\{\,(\,\overast{D}_{hX}\,\overast{R})(Y,Z)\}=0$ (Proposition \ref{bianchichern} \textbf{(e)}).
  \item[$\bullet$] The properties \textbf{(c)} and \textbf{(d)} in Proposition \ref{chenh-curv.} hold.
\end{description}
The above properties are very similar to the properties of the  Riemannian curvature. The reason lies in the condition  $\mathfrak{R}=0$ which is stronger   than the  condition \textbf{(b)} of Theorem \ref{chernconnc.}. More precisely, the condition\,\, $\overast{T}(hX,hY)=\Omega(X,Y)=0$ is stronger than the condition  $J\,\,\overast{T}(hX,hY)=0$.
}
\end{rem}
%%%%%%%%%%%%%%%%%%%%%%%%%%%%%%%%%%%%%%%%%%%%%%%%%%%%%%%%%%%%%%%%%%%%%%%%%%%%%%%%%%%%%%%%%%%%%%%%%%%%
\vspace{8pt}
\begin{center}
{\bf{\large{ Appendix:  Intrinsic Comparison}}}
\end{center}
\par The following table gives a concise comparison concerning
Berwald, Cartan and Chern connections  as well as the
fundamental geometric objects associated with them.

\begin{landscape}
 \begin{center}{\bf{Table 1: Intrinsic Comparison}}
\end{center}
\begin{center}
\small{\begin{tabular}
{|c|c|c|c|c|}\hline
&&&\\
 { Connection} &{  Berwald: \,$\overcirc{D}$ }& { Cartan:\,$D$ } &{ Chern: $\,\overast{D}$ }
\\[0.1 cm]\hline
&&&\\
{   } & $\overcirc{D}_{JX}JY=J[JX,Y]$&
 $D_{JX}JY=\overcirc{D}_{JX}JY+\mathcal{C}(X,Y)$&
$\,\overast{D}_{JX}JY=\overcirc{D}_{JX}JY$
\\[0.1 cm]
{ Expression} & $\overcirc{D}_{hX}JY=v[hX,JY]$&
$D_{hX}JY=\overcirc{D}_{hX}JY+\mathcal{C}'(X,Y)$&
$\,\overast{D}_{hX}JY=D_{hX}JY$\\
& $\overcirc{D}F=0$&${D}F=0$&$\,\overast{D} F=0$
\\[0.1 cm]\hline
&&&\\
{ $h$-torsion} & $\mathfrak{R}$& $\mathfrak{R}$& $\mathfrak{R}$
\\[0.1 cm]
{} $h$v-torsion & $0$& $\mathcal{C}'-F\mathcal{C}$& $\mathcal{C}'$
\\[0.1 cm]
{  $v$-torsion} & $0$& $0$&$0$
\\[0.1 cm]\hline
&&&\\
{ $h$-curvature} & $\overcirc{R}(X,Y)Z=(\,\overcirc{D}_{JZ}\mathfrak{R})(X,Y)$& $ R(X,Y)Z=\overcirc{R}(X,Y)Z+(D_{hX}\mathcal{C}')(Y,Z) $& $\overast{R}(X,Y)Z=R(X,Y)Z$
\\
&&$-(D_{hY}\mathcal{C}')(X,Z)
      +\mathcal{C}'(F\mathcal{C}'(X,Z),Y)$&$-\mathcal{C}(F\mathfrak{R}(X,Y),Z)$\\
    && $-\mathcal{C}'(F\mathcal{C}'(Y,Z),X)+\mathcal{C}(F\mathfrak{R}(X,Y),Z)$&\\[0.2 cm]
{ $hv$-curvature} & $\overcirc{P}(X,Y)Z=v[hX,J[JY,Z]]$& $ P(X,Y)Z =\overcirc{P}(X,Y)Z+(D_{hX}{\mathcal{C}})(Y,Z)$& $\overast{P}(X,Y)Z=\overcirc{P}(X,Y)Z$\\
&-J[JY,F[hX,JZ]]-v[h[hX,JY],JZ]&$-(D_{JY}\mathcal{C}')(X,Z)+\mathcal{C}(F\mathcal{C}'(X,Z),Y) $&$-(\,\,\overast{D}_{JY}\mathcal{C}')(X,Z)$\\
&-J[v[hX,JY],Z]&$+{\mathcal{C}}(F\mathcal{C}'(X,Y),Z)-\mathcal{C}'(F{\mathcal{C}}(Y,Z),X)$&\\
&&$-\mathcal{C}'(F{\mathcal{C}}(X,Y),Z)$&\\[0.2 cm]
{ $v$-curvature} & $0$& $ Q(X,Y)Z={\mathcal{C}}(F{\mathcal{C}}(X,Z),Y) $& $0$\\
&&$-{\mathcal{C}}(F{\mathcal{C}}(Y,Z),X)$&\\\hline
{ $v$-metricity}& not $v$-metrical & $v$-metrical &
not $v$-metrical
\\[0.1 cm]
{$ h$-metricity}& not $h$-metrical&  $h$-metrical&  $h$-metrical
\\[0.1 cm]\hline
\end{tabular}}
\end{center}
\end{landscape}
%%%%%%%%%%%%%%%%%%%%%%%%%%%%%%%%%%%%%%%%%%%%%%%%%%%%%%%%%%%%%%%%%%%%%%%%%%%%%%%%%%%%%%%%%%%%%%%%%%%%%%%%%%%%%%%%
\providecommand{\bysame}{\leavevmode\hbox
to3em{\hrulefill}\thinspace}
\providecommand{\MR}{\relax\ifhmode\unskip\space\fi MR }
% \MRhref is called by the amsart/book/proc definition of \MR.
\providecommand{\MRhref}[2]{%
  \href{http://www.ams.org/mathscinet-getitem?mr=#1}{#2}
} \providecommand{\href}[2]{#2}

\end{document}